\documentclass{amsart}% option [draft] deactivates hyperlinks (in DVI)

\usepackage[mathscr]{eucal}% So-called Euler fonts, allows \mathscr
\usepackage{amssymb}% allows special arrows like >-> e.g.
\usepackage[usenames,dvipsnames]{color}
\usepackage[normalem]{ulem}% allows \sout to strike-out, cross-out text.
\usepackage{amsthm}% allows theorem* , e.g.
\usepackage{bbold}% allows \unit \mathbb{1}
\usepackage{enumerate}% allows easy change of label
\usepackage{array}% allows array
\usepackage{amsmath}% allows pmatrix

\usepackage{hyperref}

\hypersetup{colorlinks=true,citecolor=Brown,urlcolor=Brown,linkcolor=Brown}% quote away to generate ps% replace BrickRed by Brown?

\numberwithin{equation}{section}% makes equat numb contain the section
\usepackage[all]{xy}
\newdir{ >}{{}*!/-10pt/\dir{>}}

\swapnumbers %pour que le numéro apparaisse devant le théorème

\newtheorem{Thm}[equation]{Theorem}
\newtheorem*{Thm*}{Theorem}

\newtheorem{Lem}[equation]{Lemma}

\theoremstyle{remark}
\newtheorem{Cons}[equation]{Construction}

\newtheorem{Def}[equation]{Definition}

\newtheorem{Exa}[equation]{Example}

\newtheorem{Rem}[equation]{Remark}

\newcommand{\nc}{\newcommand}
\nc{\dmo}{\DeclareMathOperator}

\dmo{\Abelem}{Abelem}
\dmo{\Ab}{Ab}
\dmo{\Add}{Add}
\dmo{\Coind}{CoInd}
\dmo{\cone}{cone}
\dmo{\Der}{D}% ground notation for derived categories
\dmo{\End}{End}
\dmo{\Ext}{Ext}
\dmo{\Hom}{Hom}
\dmo{\Ho}{Ho}
\dmo{\Id}{Id}
\dmo{\id}{id}
\dmo{\Img}{Im}
\dmo{\Ind}{Ind}
\dmo{\Ker}{Ker}
\dmo{\KK}{KK}
\dmo{\Mod}{Mod}% sheaves of modules
\dmo{\Obj}{Obj}
\dmo{\opname}{op}
\dmo{\Or}{Or}
\dmo{\proj}{proj}% projective modules
\dmo{\Res}{Res}
\dmo{\rmH}{H}
\dmo{\rmL}{\mathsf{L}}
\dmo{\rmR}{\mathsf{R}}
\dmo{\Rname}{R}
\dmo{\Shv}{Shv}% sheaves
\dmo{\SH}{SH}% ground name for cat of spectra
\dmo{\smallb}{b}% ground exponent for ``bounded''
\dmo{\Spc}{Spc}
\dmo{\Spec}{Spec}
\dmo{\supp}{supp}

\nc{\Beren}[1]{{\color{MidnightBlue}#1}}
\nc{\Ivo}[1]{{\color{OliveGreen}#1}}
\nc{\Paul}[1]{{\color{BlueViolet}#1}}

\nc{\adjto}{\rightleftarrows}
\nc{\aka}{{a.\,k.\,a.}\ }
\nc{\Alg}{\mathrm{Alg}}
\nc{\bbA}{\mathbb{A}}
\nc{\bbC}{\mathbb{C}}
\nc{\cat}[1]{\mathscr{#1}}%or: \nc{\cat}[1]{\mathcal{#1}}
\nc{\colim}{\mathop{\mathrm{colim}}}
\nc{\Cont}{\mathrm{C}} % for continuous functions
\nc{\cO}{\mathcal{O}}% for structure sheaves
\nc{\csbull}{\raisebox{1.25pt}{${\scriptscriptstyle\bullet}$}}
\nc{\eg}{{\sl e.g.}}
\nc{\eps}{\epsilon}
\nc{\hook}{\hookrightarrow}
\nc{\Idcat}[1]{\Id_{\cat{#1}}}
\nc{\ie}{{\sl i.e.}}
\nc{\into}{\mathop{\rightarrowtail}}
\nc{\inv}{^{-1}}
\nc{\isoto}{\buildrel \sim\over\to}
\nc{\LL}{\rmL\!}
\nc{\Mid}{\,\big|\,}
\nc{\MModcat}[1]{\MMod_{\cat #1}}%
\nc{\MMod}{\,\text{-}\Mod}%
\nc{\onto}{\mathop{\twoheadrightarrow}}
\nc{\op}{^{\opname}}
\nc{\otoo}[1]{\overset{#1}{\,\too\,}}
\nc{\RR}{\rmR\!}
\nc{\SET}[2]{\big\{\,#1\Mid#2\,\big\}}
\nc{\smashh}{\wedge}
\nc{\suppcat}[1]{\supp(\cat #1)}
\nc{\too}{\mathop{\longrightarrow}\limits}
\nc{\unit}{\mathbb{1}}% unit for \otimes

\begin{document}

%------------------------------------------------------------------------------

\title[Restriction to subgroups as \'etale extensions]{Restriction to finite-index subgroups as \'etale extensions in topology, KK-theory and geometry}
\author{Paul Balmer}
\author{Ivo Dell'Ambrogio}
\author{Beren Sanders}
\date{February 9, 2015}

\address{Paul Balmer, Mathematics Department, UCLA, Los Angeles, CA 90095-1555, USA}
\email{balmer@math.ucla.edu}
\urladdr{http://www.math.ucla.edu/~balmer}

\address{Ivo Dell'Ambrogio, Laboratoire de Mat\'ematiques Paul Painlev\'e, Universit\'e de Lille~1, Cit\'e Scientifique -- B\^at.~M2, 59665 Villeneuve-d'Ascq Cedex, France}
\email{ivo.dellambrogio@math.univ-lille1.fr}
\urladdr{http://math.univ-lille1.fr/~dellambr/}

\address{Beren Sanders, Mathematics Department, UCLA, Los Angeles, CA 90095-1555, USA}
\email{beren@math.ucla.edu}
\urladdr{http://www.math.ucla.edu/~beren}

\begin{abstract}
	For equivariant stable homotopy theory, equivariant KK-theory and equivariant derived categories, we show  how restriction to a subgroup of finite index yields a finite commutative separable extension, analogous to finite \'etale extensions in algebraic geometry.
\end{abstract}


\thanks{First-named author supported by NSF grant~DMS-1303073.}
\thanks{Second-named author partially supported by the Labex CEMPI (ANR-11-LABX-0007-01)}

\maketitle

%------------------------------------------------------------------------------
%\centerline{\textbf{PRELIMINARY VERSION -- NOT FOR DISTRIBUTION}}
%------------------------------------------------------------------------------

\vskip-\baselineskip\vskip-\baselineskip
\tableofcontents
\vskip-\baselineskip\vskip-\baselineskip\vskip-\baselineskip

%------------------------------------------------------------------------------]]]
\section{Introduction and main results}
\medskip
%------------------------------------------------------------------------------[[[

In linear representation theory of discrete groups, the first-named author proved that restriction to a finite index subgroup can be realized as a finite \'etale extension; see \cite[Part~I]{Balmer15}.
(The exact statement is a special case of Theorem~\ref{thm:intro_DG} below.) A priori, this result of~\cite{Balmer15} seems very module-theoretic in nature. The goal of the present article is horizontal generalization to a broad array of equivariant settings, from topology to analysis. Specifically, we prove the following three results, in which the reader should feel free to assume that the group~$G$ is \emph{finite}, if so inclined.

\begin{Thm}
\label{thm:intro_SH}%
Let $G$ be a compact Lie group and let $H \leq G$ be a closed subgroup of finite index. Then the suspension $G$-spectrum $A^G_H:=\Sigma^\infty G/H_+$ is a commutative separable ring object in the equivariant stable homotopy category~$\SH(G)$. Moreover, there is an equivalence of categories $\SH(H) \cong A^G_H\MMod_{\SH(G)}$ between $\SH(H)$ and the category of left $A^G_H$-modules in~$\SH(G)$ under which the restriction functor $\SH(G) \to \SH(H)$ becomes isomorphic to the extension-of-scalars functor $\SH(G) \to A^G_H\MMod_{\SH(G)}$.
\end{Thm}

\begin{Thm}
\label{thm:intro_KK}%
Let $G$ be a second countable locally compact Hausdorff group and let $H\leq G$ be a closed subgroup of finite index. Then the finite-dimensional algebra $A^G_H:=\bbC(G/H)$ is a commutative separable ring object in the equivariant Kasparov category~$\KK(G)$ of $G$-C*-algebras. Moreover, there is an equivalence of categories $\KK(H) \cong A^G_H\MMod_{\KK(G)}$ between $\KK(H)$ and the category of left \mbox{$A^G_H$-modules} in~$\KK(G)$ under which the restriction functor $\KK(G) \to \KK(H)$ becomes isomorphic to the extension-of-scalars functor $\KK(G) \to A^G_H\MMod_{\KK(G)}$.
\end{Thm}

\begin{Thm}
\label{thm:intro_DG}%
Let $G$ be a discrete group acting on a ringed space~$S$ (for instance, a scheme) and let $H\leq G$ be a subgroup of finite index. Then the free $\cO_S$-module $A^G_H:=\cO_S(G/H)$ on~$G/H$ is a commutative separable ring object in the derived category~$\Der(G;S)$ of $G$-equivariant sheaves of $\cO_S$-modules. Moreover, there is an equivalence of categories $\Der(H;S) \cong A^G_H\MMod_{\Der(G;S)}$ between $\Der(H;S)$ and the category of left $A^G_H$-modules in~$\Der(G;S)$ under which the restriction functor $\Der(G;S) \to \Der(H;S)$ becomes isomorphic to the extension-of-scalars functor $\Der(G;S) \to A^G_H\MMod_{\Der(G;S)}$.
\end{Thm}

These theorems are proved in Sections~\ref{se:SH}, \ref{se:KK} and~\ref{se:DG} respectively.
In all three cases, the multiplication $\mu:A^G_H\otimes A^G_H\to A^G_H$ on the ring object~$A^G_H$ is
characterized by the rule\,:
\begin{equation}
\label{eq:mu}%
\mu(\gamma\otimes\gamma')=\left\{\begin{array}{cl}\gamma&\textrm{if }\gamma=\gamma'\\0&\textrm{if }\gamma\neq\gamma'
\end{array}\right. \textrm{ for all }\gamma,\gamma'\in G/H.
\end{equation}
Let us provide some explanations and motivation.

If not familiar with~\cite{Balmer15}, the reader might be surprised to see that \emph{restriction}
can be interpreted as an \emph{extension}. When we consider a category $\cat C=\cat C(G)$ depending on a group~$G$, like the above $\SH(G)$, $\KK(G)$ or $\Der(G;S)$, and when $H\leq G$ is a subgroup, the rough intuition is that \emph{the $H$-equivariant category $\cat C(H)$ should only be a ``piece" of the corresponding $G$-equivariant category~$\cat C(G)$}. At first, one might naively hope that $\cat C(H)$ is a localization of~$\cat C(G)$, as a category. Although this naive guess essentially always fails, we are going to prove that this intuition is actually valid if one uses a broader, more flexible notion of ``localization.'' This broader notion is conceptually analogous to localization with respect to the \'etale topology in algebraic geometry rather than the Zariski topology. Seen from the perspective of Galois theory, it is not so surprising that extension should be connected to restriction to a smaller group.

Let us be more precise. Consider a category~$\cat C$, like our $\cat C(G)$, equipped with a tensor $\otimes:\cat C\times \cat C \too\cat C$ and consider, as above, a \emph{ring object} $A$ in~$\cat C$ with  associative and unital multiplication $\mu:A\otimes A\to A$ (details are recalled in \S\ref{se:prepa}). The \emph{\mbox{$A$-modules} in~$\cat C$} are simply objects $x$ in $\cat C$ together with an \mbox{$A$-action} $A\otimes x\to x$ satisfying the usual rules. We can form the category $A\MModcat{C}$ of $A$-modules in~$\cat C$ and we have an \emph{extension-of-scalars} functor $F_{A}:\cat C\to A\MModcat{C}$, which maps $y$ to $A\otimes y$, as one would expect. As in commutative algebra, the ring object $A$ is said to be \emph{separable} if $\mu$ admits a section $\sigma:A\to A\otimes A$ which is $A$-linear on both sides.
A very special example of separability occurs if $\mu$ is an isomorphism (with inverse~$\sigma$) in which case the extension-of-scalars $\cat C\to A\MModcat{C}$ is just a localization of~$\cat C$. But general separable extensions are more flexible than localizations. For instance, in algebraic geometry they include finite \'etale extensions of affine schemes by~\cite[Cor.\,6.6]{Balmer11}.

Separable extensions are particularly nice for another reason, beyond the analogy with the \'etale topology; namely, they can be performed on triangulated categories without resorting to models; see~\cite{Balmer11}. Since all the above categories~$\cat C(G)$ are triangulated, our results establish a connection between these equivariant theories and the ``tensor-triangular geometry" of \'etale extensions, as initiated in~\cite{Balmer13ppb}.

Understanding restriction as an \'etale extension has already found applications in modular representation theory (see~\cite[Part~II]{Balmer15}) and it is legitimate to expect similar developments in our new examples. This will be the subject of further work.

\smallbreak

Let us say a word about our hypothesis that $G/H$ is finite. In Section~\ref{se:counterexample}, we prove the following result which shows that Theorem~\ref{thm:intro_SH} cannot hold without some finiteness assumption on~$G/H$:
\begin{Thm}
\label{thm:intro_counterex}%
	Let $G$ be a connected compact Lie group and let $H \le G$ be a non-trivial finite subgroup.
	Then the right adjoint to the restriction functor $\Res_H^G : \SH(G) \to \SH(H)$ is not faithful. In particular, $\Res^G_H$ is not an extension-of-scalars.
\end{Thm}

Let us explain what is going on. The proofs of Theorems~\ref{thm:intro_SH}--\ref{thm:intro_DG} all follow a similar pattern that we isolate in the preparatory Section~\ref{se:prepa}. In technical terms, we prove \emph{separable monadicity} of the standard restriction-coinduction adjunction and then show that the monad associated to this adjunction is given by a separable ring object. In an ideal world, for a general subgroup $H$ of a general group~$G$, we would expect the first property (monadicity) to hold when $G/H$ is discrete and the second (the ring object) when $G/H$ is compact. Then our hypothesis that $G/H$ is finite would simply result from assuming simultaneously that $G/H$ is discrete and compact. That would be the ideal treatment. However, things turn out to be more complicated, mostly due to the current state-of-development of our examples.

Firstly, $G$-equivariant stable homotopy theory is simply not developed for non-compact groups. Similarly, $G$-equivariant KK-theory, although defined for locally compact groups, lacks enough adjoints if we do not assume $G/H$ compact (see the technical reasons in Remark~\ref{Rem:IndCoInd}). These restrictions prevent a uniform treatment beyond the case of $G/H$ finite. Trying to lift those restrictions would be a massive undertaking, going way beyond the goal of the present paper. We found our results diverse enough as they are, without trying to push them into unnecessary complications. For instance, Theorems~\ref{thm:intro_SH}--\ref{thm:intro_DG} are already interesting for finite groups, where they hold unconditionally. For the same reasons, we renounced treating the $G$-equivariant derived category over non-discrete groups, \`a la~\cite{BernsteinLunts94}.

Our present goal is to show that restriction to a subgroup can be understood as an \'etale extension in a broad range of settings beyond representation theory. The above sample should provide convincing evidence of this ubiquity and should encourage our readers to try proving similar results for their favorite equivariant categories. It is very likely that future investigations will produce further examples of this phenomenon and we are confident that the method of proof presented in Section~\ref{se:prepa} will be useful for such generalizations.

%------------------------------------------------------------------------------%]]]
\medbreak
\section{General approach}
\label{se:prepa}%
\medskip
%------------------------------------------------------------------------------[[[

\subsection*{Separable monadicity}
Let us briefly recall some standard facts about monads and separability; we refer the reader to~\cite{Balmer11} and~\cite{Balmer15} for further details.
A \emph{monad} on a category $\cat C$ consists of an endofunctor $\bbA:\cat C \to \cat C$ equipped with natural transformations
$\mu :\bbA \circ \bbA \to \bbA$ and $\eta : \Id_\cat C \to \bbA$ such that $\mu$ is an associative multiplication ($\mu\circ\bbA\mu=\mu\circ\mu\bbA$) for which $\eta$ is a two-sided identity ($\mu\circ\bbA\eta=\id=\mu\circ\eta\bbA$).
An \emph{$\bbA$-module in $\cat C$} consists of a pair $(x,\rho)$ where $x$ is an object of $\cat C$ and $\rho:\bbA x \to x$ is a morphism
(the ``action'' of $\bbA$ on $x$) making the evident associativity and unit diagrams commute in $\cat C$. A \emph{morphism} of $\bbA$-modules $(x,\rho)$ and $(x',\rho')$ is a morphism $f:x \to x'$ in $\cat C$ commuting with the actions. We denote by $\bbA\MMod_\cat C$ the resulting category of modules, which is part of the \emph{Eilenberg-Moore adjunction}
\mbox{$F_\bbA :\cat C \adjto \bbA\MMod_\cat C:U_\bbA$}.
The left adjoint $F_\bbA$ sends an object $c \in \cat C$ to the \emph{free} \mbox{$\bbA$-module} $F_\bbA(c) := (\bbA c, \mu_c : \bbA \bbA c \to \bbA c)$, and
the right adjoint sends a module $(x,\rho)$ to its underlying object $U_\bbA(x,\rho) := x$.

Any adjunction $F:\cat C \adjto \cat D:U$ with unit $\eta : \Id_\cat C \to UF$ and counit ${\eps : FU \to \Id_\cat D}$
defines a monad $\bbA = (UF,U\eps F,\eta)$ on $\cat C$ and we can consider the Eilenberg-Moore adjunction associated with this monad as above.
There is a unique ``comparison'' functor $E:\cat D \to \bbA\MMod_\cat C$
such that $E\circ F = F_\bbA$ and $U_\bbA \circ E = U$
\begin{equation} \label{Eq:EM_monad}
\begin{gathered}
\xymatrix{
& \cat C \ar@<-2pt>[dl]_F  \ar@<-2pt>[dr]_(.4){F_\bbA} & \\
\cat D \ar@<-2pt>[ur]_(.6)U \ar[rr]_-E && \bbA \MMod_\cat C \ar@<-2pt>[ul]_{U_\bbA}
}
\end{gathered}
\end{equation}
and we say that the adjunction $F\dashv U$ is \emph{monadic} if the comparison functor \mbox{$E: \cat D \to \bbA\MMod_\cat C$} is an equivalence of categories.
Concretely, $E$ is given by $E(d)= (Ud, U\eps_d : UFUd\to Ud)$ on objects $d\in \cat D$ and by~$U(f)$ on morphisms $f: d\to d'$.

We stress that, although the construction of~$E$ is formal, the \emph{property} that it is an equivalence is highly non-trivial and simply fails in general. At the extreme, taking $\cat D$ arbitrary and $\cat C=0$ (hence $\bbA\MModcat{C}=0$) shows that $E$ can get as bad as one wants. Hence, monadicity is a non-trivial property. Note that since $U_\bbA$ is faithful, a necessary condition for $E$ to be an equivalence is faithfulness of~$U$.

A monad $\bbA :\cat C \to \cat C$ is said to be \emph{separable} if
the multiplication $\mu : \bbA \circ \bbA \to \bbA$ admits a natural section $\sigma: \bbA \to \bbA \circ \bbA$ which is $\bbA,\bbA$-linear: $\mu \bbA \circ \bbA \sigma = \sigma \circ \mu = \bbA \mu \circ \sigma \bbA$.

\begin{Lem}\label{Lem:sep_mon}
Let $F: \cat C \leftrightarrows \cat D : U$ be an adjunction between idempotent-complete additive categories, and assume that the counit $\eps: FU\to \Id_\cat D$
admits a section, \ie, a natural morphism $\xi: \Id_\cat D\to FU$ such that $\eps\circ\xi=\id$.
Then the adjunction is \emph{separably monadic}. That is, the monad $UF$ on $\cat C$ is separable and the Eilenberg-Moore comparison functor~$E$ in~\eqref{Eq:EM_monad} is an equivalence.
A quasi-inverse $E\inv : \bbA\MMod_\cat C \stackrel{\sim}{\to} \cat D$ is obtained by sending $(x,\rho)\in \bbA\MMod_\cat C$ to the image $E\inv(x,\rho) := \Img(e)$ of the idempotent $e^2=e:= F(\rho) \circ \xi_{F(x)}$ on~$F(x)\in \cat D$.
\end{Lem}

\begin{proof}
The fact that $E$ is an equivalence is \cite[{Lemma 2.10}]{Balmer15}.
In order to show that the described $E\inv$ is quasi-inverse to~$E$, it suffices to show that it is a well-defined functor and that $E\inv E\cong \Id_{\cat D}$.
The latter is a straightforward verification.
For the former, it suffices to show that $e=F(\rho) \circ \xi_{F(x)}$ is idempotent; then its image exists because~$\cat C$ is idempotent-complete, and the assignment $(x,\rho)\mapsto \Img(e)$ extends to a well-defined functor in the evident way, by sending $f:(x,\rho)\to (x', \rho')$ to $e'F(f)e: \Img(e)\to \Img(e')$.
To see why $e^2=e$, consider the following diagram:
\[
\xymatrix{
F(x) \ar[d]_-{\xi_{Fx}} \ar[r]^-{\xi_{Fx}} &
 FUF(x) \ar[d]_-{FU(\xi_{Fx})} \ar[dr]^-{\id} & \\
FUF(x) \ar[d]_-{F(\rho)} \ar[r]_{\xi_{FUFx}} &
 FUFUF(x) \ar[r]_-{FU(\eps_{Fx} )} \ar[d]_-{FUF(\rho)} &
  FUF(x) \ar[d]^-{F(\rho)}   \\
F(x) \ar[r]_-{\xi_{Fx}} &
 FUF(x) \ar[r]_-{F(\rho)} & F(x).
}
\]
The two left squares commute by naturality of~$\xi$, the right square because~$\rho$ is an action and the triangle because $\xi$ is a section of~$\eps$. The perimeter reads~$e^2=e$.
\end{proof}

\begin{Exa}\label{Ex:ring}
Assume that $(\cat C, \otimes, \unit)$ is a tensor category, by which we mean a symmetric monoidal category with tensor $\otimes$ and unit object~$\unit$. Let $A=(A,\mu, \iota)$ be a \emph{ring object} (\aka \emph{monoid}) in~$\cat C$, that is, an object $A\in \cat C$ equipped with a \emph{multiplication} $\mu: A\otimes A\to A$ and \emph{unit} $\iota: \unit \to A$ such that the associativity axiom $\mu (\mu \otimes \id_A)= \mu (\id_A \otimes \mu)$ and unit axioms $\mu  (\iota \otimes \id_A)= \id_A = \mu (\id_A \otimes \iota)$ hold in~$\cat C$.
Then~$A$ defines a monad $\bbA= A \otimes -: \cat C\to \cat C$ with multiplication $\mu\otimes -$ and unit $\iota\otimes -$ (adjusted by the associativity and unit constraints of~$\otimes$).
In this case, we use the notation $F_A : \cat C\leftrightarrows A\MMod_\cat C: U_A$ for the resulting Eilenberg-Moore adjunction and call $F_A$ the \emph{extension-of-scalars} functor, as in the Introduction.
Thus an $A$-module $(x,\rho)\in A\MMod_\cat C$ consists of an object $x\in \cat C$ equipped with a map $\rho: A\otimes x\to x$ such that $\rho(\mu \otimes \id_x)= \rho(\id_A\otimes \rho)$ and $\rho (\iota\otimes \id_x)=\id_x$ in~$\cat C$.
If the multiplication $\mu:A \otimes A \to A$ admits an $A,A$-linear section $\sigma:A\to A\otimes A$
then $A$ is said to be \emph{separable}. In this case the associated monad $A \otimes-$ will be a separable monad.
\end{Exa}

\subsection*{The projection formula}
Assume that both $\cat C$ and~$\cat D$ are tensor categories and that $F:\cat C\to \cat D$ is a tensor functor ($=$~strong symmetric monoidal functor), \ie~it comes with coherent isomorphisms $\unit\stackrel{\sim}{\to} F(\unit)$ and $F(x)\otimes F(y)\stackrel{\sim}{\to} F(x\otimes y)$. A right adjoint~$U$ of~$F$ inherits the structure of a \emph{lax} tensor functor, consisting of coherent maps $\iota:\unit\to U(\unit)$ and $\lambda:U(x)\otimes U(y)\to U(x\otimes y)$. They are defined by
\[
\iota:
\xymatrix@C=2em{
\unit \ar[r]^-\eta & UF(\unit) \ar[r]^-\sim & U(\unit)
}
\]
\begin{equation}
\label{eq:lambda}%
\lambda:
\xymatrix@C=1.4em{
U(x)\otimes U(y) \ar[r]^-{\eta}
& UF(Ux \otimes Uy) \ar[r]^-{\sim}
& U(FUx \otimes FUy) \ar[rr]^-{U(\eps\,\otimes\,\eps)}
&& U(x\otimes y)
}
\end{equation}
and they are not necessarily invertible. Lax monoidal functors preserve ring objects. In particular, we obtain a commutative ring object $A:=(U(\unit), \mu, \iota)$  in~$\cat C$, where we endow~$U(\unit)$ with the unit map~$\iota$ as above and the multiplication
\begin{equation} \label{eq:unit_monoid}
\mu: \xymatrix@C=3em{
U(\unit) \otimes U(\unit) \ar[r]^-{\lambda} & U(\unit \otimes \unit) \ar[r]^-{\sim} & U(\unit) \,.
}
\end{equation}
The lax monoidal structure of~$U$ also defines a natural transformation
\begin{equation}\label{eq:right_projection_formula}
\xymatrix@C=3em{ \pi : U(y) \otimes x \ar[r]^-{\id \otimes\, \eta} & U(y) \otimes UF(x) \ar[r]^-{\lambda} & U(y \otimes F(x))}\end{equation}
for all $x\in\cat C$ and $y \in\cat D$, which we call the \emph{projection morphism}.

\begin{Def} \label{Def:projection_formula}
We say that the \emph{projection formula holds} for the adjunction $F \dashv U$ when the natural morphism~$\pi:U(x)\otimes y\to U(x\otimes F(y))$ of~\eqref{eq:right_projection_formula} is an isomorphism for all $x \in \cat{C}$ and $y\in \cat{D}$.
\end{Def}

Even when the morphism $\pi$ is not an isomorphism, it compares \emph{two} monads on~$\cat C$: the monad $A\otimes(-)$ induced by the ring object $A=(U\unit, \mu, \iota)$ and the monad $\bbA=UF$ induced by the adjunction. We now prove that the natural transformation $\pi$ is always compatible with those additional structures:

\begin{Lem} \label{Lem:monad_vs_monoid}
With the above notation, the natural map $\pi:U(\unit)\otimes x \to U(\unit \otimes Fx)\cong UF(x)$ is a morphism $A\otimes(-)\to \mathbb A$ of monads on~$\cat C$.
\end{Lem}

\begin{proof}
We must verify that~$\pi$ identifies the units and multiplications of the two monads.
Concretely, we must show that the following diagrams commute in~$\cat C$:
\[
\xymatrix@C=1.5em{
& \unit \otimes x \cong x \ar[dl]_{\iota\, \otimes\, \id} \ar[dr]^{\eta} & \\
U(\unit)\otimes x  \ar[rr]^-{\pi} && UF (x)
}
\quad
\quad
\xymatrix@C=2em{
U(\unit)\otimes U(\unit) \otimes x \ar[d]_{\mu\, \otimes\, \id} \ar[rr]^-{\pi^{(2)}} && UF(UFx) \ar[d]^{U\eps F} \\
U(\unit ) \otimes x \ar[rr]^-{\pi} && UF(x)
}
\]
Here
$\pi^{(2)}:=(\pi_{UF x})(\id \otimes \pi_x)
= (\id\otimes \pi_x)(\pi_{U(\unit)\otimes x})$
 denotes the two-fold application of~$\pi$.
The commutativity of the above triangle follows from that of the diagram
\[\xymatrix @C=1em{
x \ar[r]^-{\sim} \ar[d]_{\eta} & \unit \otimes x \ar[d]^{\id\otimes \eta}\ar[rrr]^{\iota \,\otimes \id} &&& U(\unit) \otimes x \ar[d]^{\id\otimes\, \eta} \ar@/^1pc/[drrrr]^{\pi} &&& \\
UFx \ar[r]^-{\sim} & \unit \otimes UF(x) \ar[rrr]^{\iota\,\otimes \,\id} &&& U(\unit) \otimes UF(x)
\ar@{}[urrr]|{\textrm{def.}}
\ar[rrr]^-\lambda &&& U(\unit \otimes Fx ) \ar[r]^-{\sim} & UFx
		}\]
once we note that the bottom row is the identity; the latter holds because the following diagram commutes (since $U$ is lax monoidal)
	\[\xymatrix@C=4em{ \unit \otimes Uy \ar[d]_-{\sim} \ar[r]^-{\iota\,\otimes\, \id} & U\unit \otimes Uy \ar[d]^\lambda \\
		Uy \ar[r]^-{\sim} & U(\unit \otimes y)
		}\]
for all~$y$ (plug $y:=Fx$). Next we check that the following diagram commutes
\[\xymatrix{
		 U\unit \otimes U\unit \otimes x \ar[d]^{\mu\,\otimes \,\id} \ar[r]^-{\id\,\otimes \,\pi} & U\unit \otimes  UFx  \ar[r]^{\pi} \ar[rd]^-{\lambda} & UFUFx \ar[d]^{U\eps F} \\
		U\unit \otimes x \ar[rr]^-{\pi} && UFx
	}\]
by using the definition of~$\pi=\lambda\,(\id\otimes\eta)$ in~\eqref{eq:right_projection_formula}, commutativity of the diagram
\[\xymatrix@C=4em{
		U\unit \otimes U\unit \otimes x  \ar[d]^{\mu\,\otimes\,\id}\ar[r]^-{\id\otimes\id\otimes\eta} &
		 U\unit\otimes U\unit \otimes UFx  \ar[d]^{\mu\,\otimes\,\id} \ar[r]^-{\id \otimes \lambda} &
		  U\unit \otimes UFx \ar[d]^{\lambda} \\
		U\unit \otimes x  \ar[r]^-{\id \otimes \eta } & U\unit \otimes UFx  \ar[r]^-{\lambda} & UFx
	}\]
which reads $\pi\,(\mu\otimes\id)=\lambda(\id\otimes\pi)$, and commutativity of the diagram
\[\xymatrix@C=4em{
		U\unit \otimes UFx  \ar[dr]_{\id} \ar[r]^-{\id \otimes \eta UF} &
		U\unit \otimes UFUFx  \ar[d]^{\id \otimes U\eps F } \ar[r]^-{\lambda} &
		UFUFx \ar[d]^{U\eps F} \\
		& U\unit \otimes UFx  \ar[r]^-{\lambda} & UFx
	}\]
which reads $\lambda=U\eps F\,\pi$. Here we have suppressed unital isomorphisms for readability.
\end{proof}

We now prove an analogue of Beck Monadicity in our framework\,:
\begin{Thm}\label{Thm:general}
Let $F: \cat C\leftrightarrows \cat D:U$ be an adjunction of idempotent-complete additive tensor categories, where~$F$ is a tensor functor. Assume moreover that:
\begin{itemize}
	\item[(a)] The counit of the adjunction $\eps:FU\to \Idcat{D}$ admits a natural section.
\item[(b)] The projection formula holds $U(x)\otimes y\cong U(x\otimes F(y))$; see Definition~\ref{Def:projection_formula}.
\end{itemize}
Then the adjunction is monadic and the associated monad is isomorphic to the one induced by the commutative ring object $A=(U(\unit), \mu,\iota)$ in~$\cat C$; see~\eqref{eq:unit_monoid}.
Thus there is a (unique) equivalence~$E:\cat D\stackrel{\sim}{\to} A\MMod_\cat C$ identifying the given adjunction $F\dashv U$ with the free-forgetful adjunction $F_A\dashv U_A$, up to isomorphism; see~\eqref{Eq:EM_monad}.
\end{Thm}

\begin{proof}
Just combine Lemmas~\ref{Lem:sep_mon} and~\ref{Lem:monad_vs_monoid}.
Explicitly, the equivalence $E:\cat D \isoto U(\unit)\MMod_\cat C$ sends
	$d \in \cat D$ to the $U(\unit)$-module $E(d) := U(d)$ with action
	given by
	$U(\unit) \otimes U(d) \otoo{\lambda} U(\unit \otimes d) \simeq Ud$.
\end{proof}

Recall now the situation of the Introduction. In each of the three examples discussed there, we are given suitable groups~$G$ and~$H$, where $H$ is a subgroup of finite index in~$G$. We have associated tensor triangulated categories $\cat C:= \cat C(G)$ and $\cat D:=\cat C(H)$ and a tensor exact restriction functor $F:=\Res^G_H: \cat C(G)\to \cat C(H)$ admitting a right adjoint $U:=\Coind_H^G : \cat C(H)\to \cat C(G)$.
In the remaining three sections, we are going to deduce Theorems~\ref{thm:intro_SH}, \ref{thm:intro_KK} and \ref{thm:intro_DG} from Theorem~\ref{Thm:general}, by verifying in each case that the hypotheses~(a) and~(b) hold. We will sometimes abbreviate the corresponding monad as~$\bbA=\bbA^G_H:=\Coind_H^G\circ\Res^G_H$.

\begin{Rem}
\label{rem:sep}%
Although the splitting of the counit implies (cf.~Lemma~\ref{Lem:sep_mon}) that the monad
$UF \simeq U(\unit) \otimes (-)$ is separable,
it does not follow \emph{a priori} that
the ring object $U(\unit)$ is itself separable. However, in all examples, the ring object $A=A^G_H\simeq U(\unit)$ will be separable with an obvious $A,A$-linear section~$\sigma:A^G_H\to A^G_H\otimes A^G_H$ of its multiplication, which will always be induced by the diagonal map $G/H\to G/H\times G/H$.

This is clearly the simplest explanation for separability of~$U(\unit)$ in the examples. However, one could also expand the above abstract treatment to obtain separability of~$U(\unit)$ from general arguments. Indeed, in all our examples, the functor $F$ also has a \emph{left} adjoint $L$ which is isomorphic to $U$ and the
section $x \to FU (x)$ we construct of the counit of $F \dashv U$ coincides with the
\emph{unit} $x \to FL (x)$ of the $L\dashv F$ adjunction under this isomorphism~$L\simeq U$.
Using compatibility of left and right projection formulas, one can then show that $U(\unit)$ is indeed separable in that case. Further details are left to the interested reader.
\end{Rem}

%------------------------------------------------------------------------------%]]]
\bigbreak
\section{Equivariant stable homotopy theory}
\label{se:SH}%
%------------------------------------------------------------------------------[[[

Let $G$ be a compact Lie group and let $\SH(G)$ denote the stable homotopy category of (genuine) $G$-spectra, in the sense of \cite{LMS86}.
This is a compactly generated tensor triangulated category with the smash product of $G$-spectra, $-\smashh -$, and unit $S=\Sigma^\infty S^0$.
For any closed subgroup $H \le G$, we have a \emph{restriction} tensor functor $\Res_H^G : \SH(G) \to \SH(H)$ which admits a left adjoint, \emph{induction}, denoted $G_+ \smashh_H -$, and a right adjoint, \emph{coinduction}, denoted $F_H(G_+,-)$.
The two adjoints are related by the \emph{Wirthm\"uller isomorphism}, which is a natural isomorphism
\begin{equation}\label{eq:wirthmuller_isomorphism}
	\omega_X : F_H(G_+,X) \isoto G_+ \smashh_H (X \smashh S^{-L}) \,,
\end{equation}
defined for any $H$-spectrum $X$, where $L$ is the tangent $H$-representation of $G/H$ at the identity coset $eH$.
(See~\cite[\S XVI.4]{May96} and \cite{May03}.)
Note that if $H$ has finite index in $G$ then $L=0$
and in this case $\omega_X$ provides an
isomorphism between the induction and coinduction functors:
$F_H(G_+, X) \isoto G_+ \smashh_H X $.
\begin{Lem} \label{Lem:proj_SH}
The restriction-coinduction adjunction
$\Res_H^G: \SH(G) \leftrightarrows \SH(H) : F_H(G_+, -)$ between the equivariant stable homotopy categories satisfies the projection formula (Def.\,\ref{Def:projection_formula}).
\end{Lem}
\begin{proof}
As noted in \cite[Rem.\,1.2]{May03}, this follows from the Wirthm\"uller isomorphism. We briefly recall the argument. By \cite[Prop.\,3.2]{FauskHuMay03} the projection map
\[
\pi : Y \smashh F_H( G_+,X) \longrightarrow F_H(G_+, \Res^G_H(Y)\smashh X)
\]
as in~\eqref{eq:right_projection_formula} is invertible whenever $X$ is a dualizable (\ie\ compact) object of~$\SH(H)$.
Smashing with any object preserves coproducts. Hence so does $F_H(G_+,-)$, by the Wirthm\"uller isomorphism~\eqref{eq:wirthmuller_isomorphism}.
Thus the two functors (in~$X$) that are compared by~$\pi$ are both exact and commute with coproducts.
Since $\SH(H)$ is compactly generated, it follows that $\pi$ is invertible for arbitrary $X$ as well.
\end{proof}
\begin{Lem}\label{Lem:splitting_SH}
	Let $G$ be a compact Lie group and let $H \le G$ be a closed subgroup of finite index in $G$.
	Then the counit of the restriction-coinduction adjunction
between $\SH(G)$ and $\SH(H)$ has a natural section.
\end{Lem}
\begin{proof}
Since $G/H$ is finite, hence discrete, the $H$-space $G_+$ decomposes as a coproduct \mbox{$(G-H)\sqcup H_+$} and we can define for any based $H$-space $X$ a \emph{continuous} $H$-equivariant map $\xi_X : X \to F_H(G_+,X)$ by
\[\xi_X(x)(g) = \begin{cases} gx & \text{if } g\in H\\
			* & \text{if } g\notin H.
		\end{cases}\]
This map is natural in $X$ and defines a section of the counit
\[\begin{array}{cccc}
\eps_X :& F_H(G_+,X) & \too & X
\\
& \varphi & \longmapsto & \varphi(1)
\end{array}
\]
of the space-level restriction-coinduction adjunction.
At the level of spectra, recall from \cite[\S II.4, p.~77]{LMS86} that coinduction is defined spacewise
without the need to spectrify---that is, for an $H$-spectrum $D$, the $G$-spectrum $F_H(G_+,D)$ is defined for every $G$-representation~$V$ in the indexing universe by $F_H(G_+,D)(V) := F_H(G_+,D(V))$ where the right-hand side is the space-level coinduction functor.
One checks that the maps $\xi_{D(V)}$ define a map of \mbox{$H$-spectra} $D \to F_H(G_+,D)$ by using the definition of the structure maps of $F_H(G_+,D)$ and the commutativity of
\begin{equation}\label{eq:SH_spectrum_detail}
	\begin{gathered}
	\xymatrix @C=4em{
			F(S^V,X) \ar[r]^-{\xi_{F(S^V,X)}} \ar[dr]_-{F(S^V,\xi_X)\quad} & F_H(G_+,F(S^V,X)) \ar[d]_-{\cong}^-{\phi^{-1}} \\
			& F(S^V,F_H(G_+,X)).
		}
\end{gathered}
	\end{equation}
	Here $\phi^{-1}$ is the $G$-homeomorphism defined on page 76 of \cite{LMS86} and
	the commutativity of
	diagram~\eqref{eq:SH_spectrum_detail}
	follows immediately from the definitions.
	In this way, we have constructed a section of the counit of the
	restriction-coinduction adjunction between the categories of $G$-spectra
	and $H$-spectra.
	Coinduction preserves weak equivalences so this splitting passes without difficulty
	to a splitting of the counit of the adjunction between homotopy categories: $\SH(G) \adjto \SH(H)$.
\end{proof}

\begin{proof}[Proof of Theorem~\ref{thm:intro_SH}.]
The theorem follows from Theorem~\ref{Thm:general}, since its hypotheses~(a) and~(b) have been verified in Lemmas~\ref{Lem:splitting_SH} and~\ref{Lem:proj_SH}.
It remains only to describe the ring object~$A=(A,\mu,\iota)$ in $\SH(G)$.
By definition, we have $A=F_H(G_+,S)$ and the Wirthm\"uller isomorphism identifies it with the $G$-equivariant suspension spectrum $G_+\smashh_H S \cong \Sigma^\infty(G/H)_+$. Under this isomorphism, the multiplication in~$F_H(G_+,S)$ becomes the multiplication $A^G_H\otimes A^G_H\to A^G_H$ announced in the Introduction. Under the isomorphism $\Sigma^\infty(G/H)_+\smashh\Sigma^\infty(G/H)_+\cong\Sigma^\infty(G/H\times G/H)_+$ it is given by the map $(\gamma,\gamma')\mapsto \gamma$ when $\gamma=\gamma'$ and $*$ otherwise; see~\eqref{eq:mu}. An obvious section~$\sigma$ is given by $\gamma\mapsto (\gamma,\gamma)$, showing that $A^G_H$ is indeed separable as a ring object. See Remark~\ref{rem:sep}.
\end{proof}

%------------------------------------------------------------------------------%]]]
\bigbreak
\section{Equivariant KK-theory}
\label{se:KK}%
%------------------------------------------------------------------------------[[[

We begin with some recollections on equivariant KK-theory. Details can be found in~\cite{Meyer08}
and the references therein.

Let $G$ be a second countable locally compact Hausdorff group.
For short, we use the term \emph{$G$-algebra} to mean a (topologically) separable\footnote{A C*-algebra is \emph{separable} if it admits a countable subset which is dense (for the norm topology). This is not related to the separability of monads and rings discussed in Section~\ref{se:prepa}.} complex C*-algebra equipped with a continuous left $G$-action $G\times A\to A$ by $*$-isomorphisms.
We denote by $\Alg(G)$ the category of $G$-algebras and $G$-equivariant $*$-homomorphisms (\emph{morphisms} for short).
It is a symmetric monoidal category when equipped with the minimal tensor product $-\otimes -$ (\ie, the completion of the algebraic tensor product $-\otimes_\mathbb C-$ with respect to the minimal C*-norm), where $G$ acts diagonally on tensor products: $g(a\otimes b):= ga \otimes gb$.

The \emph{$G$-equivariant Kasparov category}, here denoted $\KK(G)$, has the same objects as $\Alg(G)$ and its Hom sets are Kasparov's $G$-equivariant bivariant K-theory groups $\Hom_{\KK(G)}(A,B)=KK^G(A,B)$ with composition given by the so-called Kasparov product.
It is a tensor
triangulated category admitting all countable coproducts, so in particular it
is an idempotent-complete additive category.  It comes equipped with a
canonical tensor functor $\Alg(G)\to \KK(G)$ which is the identity on objects.
We will not distinguish notationally between a morphism of $\Alg(G)$ and its canonical image in~$\KK(G)$.

\begin{Cons}\label{Cons:KK}
Let $H$ be a closed subgroup of~$G$ and assume already, for simplicity, that the quotient~$G/H$ is a finite discrete space (see Remark~\ref{Rem:IndCoInd} for more general results).
By restricting $G$-actions to~$H$, we obtain a \emph{restriction functor} $\Res^G_H\colon \Alg(G)\to \Alg(H)$.
Define \emph{coinduction}  $\Coind_H^G: \Alg(H)\to \Alg(G)$ as follows.
Its value on an $H$-algebra~$B$ is the $G$-algebra
\[
\Coind_H^G(B):= \Cont_b(G,B)^H
\]
of those bounded continuous functions $f:G\to B$ that are $H$-invariant: $f(hx)=hf(x)$ for all $h\in H$ and $x\in G$.
This is again a separable C*-algebra with the supremum norm and the pointwise algebraic operations.
The left $G$-action is given by $(g\cdot f)(x) := f(xg)$.
The functoriality is obtained by composing functions with morphisms $B\to B'$.
For all $A\in \Alg(G)$ and $B\in \Alg(H)$, we define the unit and counit natural transformations by the ``usual" formulas:
\begin{align*}
& \eta_A : A \to \Coind_H^G \Res^G_H(A) ,
 \quad a \mapsto \eta_A(a):=(G \ni t\mapsto ta \in A)   \\
& \eps_B  : \Res^G_H \Coind_H^G(B) \to B ,
 \quad f\mapsto \eps_B(f):=f(1)
\end{align*}
for $a\in A$ and $f\in \Coind_H^G(B)$.
\end{Cons}

\begin{Lem}
\label{lem:xi-proj-KK}%
The functor $\Res^G_H:\Alg(G)\to \Alg(H)$ is left adjoint to $\Coind_H^G:\Alg(H)\to \Alg(G)$, with the above $\eta$ and $\eps$ as unit and counit. Moreover:
\begin{itemize}
\item[(a)] The counit $\eps$ admits a natural section.
\item[(b)] The adjunction satisfies the projection formula (Def.\,\ref{Def:projection_formula}).
\end{itemize}
\end{Lem}

\begin{proof}
This adjunction is well-known (cf.\ \cite[{p.\,231}]{MeyerNest06} or \cite[{Prop.\,38}]{Meyer11}) and works equally fine with $G/H$ compact, although we could not locate the explicit unit and counit in the literature. The verifications that the above~$\eta$ and~$\eps$ are well-defined and satisfy the triangle equalities are immediate.

Since $G/H$ is assumed discrete, for each $b\in B$ the formula
\begin{equation}
\label{eq:xi-KK}%
G\ni t\; \mapsto \;\xi_B(b)(t) := \left\{
\begin{array}{cl}
tb & \textrm{ if }t\in H
\\
0 & \textrm{ if }t\notin H
\end{array}
\right.
\end{equation}
yields a well-defined map $\xi_B(b)\in \Cont_b(G, B)^H$.
Moreover, the assignment $b \mapsto \xi_B(b)$ defines an $H$-equivariant morphism $\xi_B: B\to \Cont_b(G, B)^H$ and therefore a natural transformation $\xi: \Id_{\Alg(H)}\to \Res^G_H\Coind^G_H$.
Since $\eps_B \xi_B (b)= \xi_B(b)(1)=1b=b$ for all $b\in B$, we see that~$\xi$ provides the natural section claimed in part~(a).

Let us prove part~(b).
By unfolding the definitions, we see that the present incarnation of the projection map~\eqref{eq:right_projection_formula} is the following morphism of $G$-algebras:
\begin{align*}
\pi: \quad \Coind_H^G(B)\otimes A & \longrightarrow \Coind_H^G(B \otimes \Res_H^G(A)) \\
 f\otimes a
 \quad & \longmapsto \quad
   \pi(f\otimes a)=\big(G \ni t \mapsto f(t) \otimes ta \big)
\end{align*}
for all $B\in \Alg(H)$ and $A\in \Alg(G)$. The above formula defines~$\pi$ on simple algebraic tensors, and extends uniquely to the minimal tensor product by linearity and continuity.
Once again, the fact that this is an isomorphism is well-known to the experts but the details are hard to find in the literature. For $G/H$ finite it is actually easy. Choose a full set $R\subset G$ of representatives modulo~$H$.
Then, forgetting actions, the inclusion $R\hookrightarrow G$ induces a natural (!) isomorphism of (non-equivariant) C*-algebras $\rho: \Coind_H^G(D)\stackrel{\sim}{\to} \prod_{r\in R} D$ for all $D\in \Alg(H)$, by $\rho(f) = (f(r))_r$.
We thus obtain the following commutative square of C*-algebras
\[
\xymatrix@R=1.5em{
\Coind_H^G (B) \otimes A \ar[r]^-{\rho \otimes \id}_-\cong  \ar[d]_-{\pi} &
 \left( \prod_{r\in R} B \right) \otimes A \ar[d]  \\
\Coind_H^G(B \otimes \Res^G_H A) \ar[r]^-{\rho}_-\cong &
  \prod_{r\in R} (B \otimes A)
}
\]
where the right vertical map is defined, on simple tensors, by $(b_r)_r\otimes a \mapsto (b_r\otimes ra)_r$.
The latter is invertible because $R$ is finite, hence so is~$\pi$.
\end{proof}

\begin{Rem}\label{Rem:IndCoInd}
If $G/H$ is compact, but not necessarily discrete, Construction \ref{Cons:KK} still works verbatim to provide the right adjoint of restriction, but some non-trivial analysis (using that $H$ acts properly on~$G$) is needed to prove that the functor yields separable C*-algebras and satisfies the projection formula.
When $H$ is any closed subgroup, without any hypothesis on~$G/H$, it is convenient to consider an
 ``induction'' functor $\Ind_H^G: \Alg(H)\to \Alg(G)$ defined by the subalgebra $\Ind_H^G (B) := \{f\in \Cont_b(G,B)^H\mid (t\mapsto \|f(t)\|) \in \Cont_0(G/H)\}$ of $\Coind_H^G(B)$.
 The same analytical arguments apply to show that $\Ind_H^G$ takes separable values and satisfies the projection formula, in full generality.
 Clearly $\Ind_H^G=\Coind_H^G$ when $G/H$ is compact, so it is right adjoint to~$\Res_H^G$ in this case. If instead $G/H$ is discrete then $\Ind_H^G$ becomes, at the level of the Kasparov categories, \emph{left} adjoint to restriction, thus earning its name; see \cite[\S2.6]{Meyer11}.
 It is not known whether there exists, unconditionally, a left or right adjoint to restriction on the Kasparov categories. In our view, this anomalous behavior --- quite unlike the situation in representation theory or equivariant stable homotopy --- is a cost of the endemic countability hypotheses required by the analytical constructions traditionally involved with KK-theory, which prevent the Kasparov categories from admitting arbitrary small  coproducts.
\end{Rem}

\begin{proof}[Proof of Theorem~\ref{thm:intro_KK}]
We are assuming that $G/H$ is finite, hence discrete, so we can apply Lemma~\ref{lem:xi-proj-KK}.
We claim that the conclusions of Lemma~\ref{lem:xi-proj-KK} still hold at the level of KK-theory, not only at the level of algebras and algebra morphisms.

Indeed, recall that the canonical functor $\Alg(G)\to \Alg(G)[W_G^{-1}]= \KK(G)$ is a localization of categories, obtained by inverting precisely the set~$W_G$ of \emph{$G$-equivariant KK-equivalences}, and similarly for~$H$.
This follows immediately from Meyer's universal property of equivariant KK-theory~\cite{Meyer00}, and from the following easy (and well-known) observation: each of the three properties of an additive functor $F: \Alg(G)\to \cat C$ for which $\KK(G)$ is universal --- namely, homotopy invariance, C*-stability, and split exactness --- can be expressed by the property that~$F$ sends a suitable class of morphisms of algebras to isomorphisms of~$\cat C$.

It is known that $\Res_H^G(W_G)\subset W_H$ and $\Coind_H^G(W_H)\subset W_G$ (see~\cite[\S4.1]{Meyer08}).
Hence restriction and coinduction yield well-defined functors $\Res_H^G: \KK(G)\to \KK(H)$ and $\Coind_H^G: \KK(H)\to \KK(G)$, which are again adjoint by the (canonical images of the) same unit and counit~$\eta$ and~$\eps$. Similarly, the projection isomorphism~$\pi$ and the section~$\xi$ of the counit also pass to~$\KK$.
We therefore have separable monadicity $E : \KK(H) \stackrel{\sim}{\longrightarrow} A \MMod_{\KK(G)}$ by Theorem~\ref{Thm:general},
where $A=\Coind_H^G(\bbC)=\Cont(G/H)$ is the $G$-algebra of continuous functions on~$G/H$
equipped with its pointwise multiplication and identity. This is simply the usual underlying ring structure of $\Cont(G/H)$ as a C*-algebra, which is isomorphic to its $\bbC$-linear dual algebra~$\bbC(G/H)$ as announced in Theorem~\ref{thm:intro_KK}.
Multiplication on~$A^G_H=\bbC(G/H)$ is again given by Formula~\eqref{eq:mu}. Its separability is again guaranteed by the morphism $\sigma:A^G_H\to A^G_H\otimes A^G_H$ given by $\gamma\mapsto \gamma\otimes\gamma$.
\end{proof}

\begin{Rem}
\label{rem:E-KK}%
We can describe the Eilenberg-Moore equivalence more explicitly.
By construction, the Eilenberg-Moore functor $E$ for the monad $\bbA=\Coind_H^G \Res_H^G$ is given by $ E(B)=(\Coind_H^G(B), \Coind_H^G(\eps_B))$ for all $B\in \KK(H)$.
Under the isomorphism of monads $\pi: A\otimes(-)\stackrel{\sim}{\to} \bbA$, this $\Cont(G/H)$-action becomes the morphism
\[
\xymatrix@1{
 \Cont(G/H)\otimes \Coind_H^G(B) \ar[r]^-{\pi}_-{\sim} &
\Coind_H^G \Res^G_H \Coind_H^G ( B)
\ar[rr]^-{\Coind_H^G(\eps_B)} &&
\Coind_H^G(B)
}
\]
which sends $u\otimes f$ to the function
$\eps_B \circ (x\mapsto u(x)xf)  \in \Coind^G_H(B)$, \ie, to the product $ uf = (x \mapsto u(x)f(x))$.
Thus the action is simply given by the pointwise multiplication of functions.
The quasi-inverse~$E^{-1}$ is described in Lemma~\ref{Lem:sep_mon}.
\end{Rem}

\begin{Rem}
An equivalence quite like our functor~$E^{-1}$ has already been described, under the name \emph{compression}, both for $G$-C*-algebras in~\cite[Lem.\,12.3 ff.]{GHT00} and also in the purely algebraic context of $
G$-rings in~\cite[\S10.3]{CortinasEllis14}. For $H$ a (finite) subgroup of a (countable) discrete~$G$, but with $G/H$ not necessarily finite, the compression functor is defined for $G$-algebras that are ``proper over~$G/H$'' and yields a quasi-inverse of induction (at least at the level of algebras). Here induction refers to a certain functor from $H$-algebras to $G$-algebras which is usually \emph{not} right adjoint to restriction.
We suspect this induction-compression equivalence is related to our Theorem~\ref{thm:intro_KK}, although the details are not yet clear to us.
\end{Rem}

%------------------------------------------------------------------------------%]]]
\bigbreak
\section{Equivariant derived categories}
\label{se:DG}%
%------------------------------------------------------------------------------[[[

Let $G$ be a discrete group, \eg\ a finite one, which acts on a (locally) ringed space~$S=(S,\cO_S)$, \eg\ a scheme. For every $g\in G$, we simply denote by $g:S\isoto S$ the corresponding isomorphism of ringed spaces, which involves compatible ring isomorphisms $g^\sharp:\cO_S(V)\isoto \cO_S(g V)$ for all $g\in G$ and $V\subseteq S$ open. For every sheaf $M$ on~$S$, the sheaf $g^*M$ is given by $g^*M\,(V)=M(g V)$. For every $g_1,g_2\in G$, we have an \emph{equality} $(g_1g_2)^*=g_2^*g_1^*$, where we could also accept an isomorphism with coherence. This equality will lighten some of the discussion below.

\begin{Def}
A \emph{$G$-equivariant sheaf of $\cO_S$-modules} is a pair $(M,\varphi)$ where $M\in\cO_S\MMod$ is a sheaf of $\cO_S$-modules and $\varphi=(\varphi_g)_g$ is a collection of $\cO_S$-linear isomorphisms
$$
\varphi_g: M\isoto g^*M
$$
indexed by~$g\in G$ such that $\varphi_{g_1g_2}=g_2^*(\varphi_{g_1})\circ \varphi_{g_2}$ for every $g_1,g_2\in G$. As usual, we often write $M$ instead of~$(M,\varphi)$ and we call $\varphi$ the ``action" of $G$ on~$M$, keeping in mind that $G$ moves the underlying sheaf.
A morphism of $G$-equivariant sheaves $f:(M,\varphi)\to (M',\varphi')$ is a morphism $f:M\to M'$ of $\cO_S$-modules such that $\varphi'_g\circ f=g^*(f)\circ \varphi_g$ for every $g\in G$. We denote by $\Shv(G;S)$ the category of \mbox{$G$-equivariant} sheaves of $\cO_S$-modules on~$S$. It is an abelian category, with a faithful exact functor $\Res^G_1:\Shv(G;S)\to \cO_S\MMod$, which forgets the $G$-equivariance.

Since $G$ is discrete, we can define the \emph{$G$-equivariant derived category of~$S$} to be the derived category $\Der(G;S):=\Der(\Shv(G;S))$ of the above abelian category.
\end{Def}

\begin{Rem}
	Everywhere below, one can replace $\cO_S$-modules by quasi-coherent ones (or coherent ones) if $S$ is a (noetherian) scheme. Similarly, one can put boundedness conditions on the derived categories, or conditions on the homology, etc. The statements remain true as long as the restriction-coinduction adjunction preserves those subcategories. Note that the ring $A^G_H$ that we are going to produce is a complex having a finite-dimensional free $\cO_S$-module concentrated in degree zero --- so it will usually belong to all such choices of subcategories.
	We leave such variations on the theme to the interested readers.

	Also, the $\cO_S$-module structure
	on the various equivariant sheaves that we will construct is not problematic, and it always comes as a second layer, once the ``sheaf-part" and the ``$G$-part" of the story are clear. We shall therefore emphasize the latter and leave most of the former as easy verifications.
\end{Rem}

\begin{Rem}
	The category $\Shv(G;S)$ is monoidal via the tensor product of $\cO_S$-modules and ``diagonal action": modulo the identification $g^*(M\otimes_{\cO_S}M')\cong g^*M\otimes_{\cO_S}g^*M'$, one defines $(\varphi\otimes\varphi')_g$ as $\varphi_g\otimes\varphi'_g$. This induces a left-derived tensor product on the derived category (that we do not really use in this glorious generality, since we really only need to tensor with the flat object~$A^G_H$).
\end{Rem}

\begin{Cons}
Let $H\leq G$ be a subgroup. We have an obvious exact restriction functor $\Res^G_H:\Shv(G;S)\to \Shv(H;S)$ which induces the equally obvious $\Res^G_H:\Der(G;S)\to \Der(H;S)$. We construct the right adjoint $\Coind_H^G:\Shv(H;S)\to \Shv(G;S)$ by explicit formulas. Let $N=(N,\psi)\in\Shv(H;S)$; in particular we have $\psi_h:N\isoto h^*N$ for every $h\in H$. For every open $V\subseteq S$, we define the $\cO_S$-module
\begin{equation}
\label{eq:Coind}%
(\Coind_H^GN)(V):=\SET{(s_t)_{t}\in\prod_{t\in G}t^*N(V)}{t^*\psi_h(s_t)=s_{ht}\ \forall\ t\in G, h\in H}
\end{equation}
whose $G$-action is given on each open by $\varphi_g\big((s_t)_{t\in G}\big):=(s_{tg})_{t\in G}$, observing that $s_{tg}\in (tg)^*N(V)=g^*t^*N(V)=t^*N(g V)$. It is easy to verify that $\Coind_H^GN$ remains a sheaf (it is a limit), \ie~we do not need sheafification. The definition of $\Coind_H^G$ on morphisms is straightforward.

We define the counit $\eps_N:\Res^G_H\Coind_H^GN\too N$ by mapping $(s_t)_{t\in G}$ to $s_1\in N(V)$ on each open $V\subseteq S$. We define the unit $\eta_M:M\to \Coind_H^G\Res^G_HM$ by mapping $m\in M(V)$ to $(\varphi_t(m))_{t\in G}$. These define natural transformations $\eta:\Id_{\Shv(G;S)}\too \Coind_H^G\Res^G_H$ and $\eps:\Res^G_H\Coind_H^G\too \Id_{\Shv(G;S)}$.
\end{Cons}

\begin{Lem}
\label{lem:xi-proj-DG}%
The functor $\Res^G_H:\Shv(G;S)\to \Shv(H;S)$ is left adjoint to $\Coind_H^G:\Shv(H;S)\to \Shv(G;S)$, with the above $\eta$ and $\eps$ as unit and counit. Moreover,
\begin{enumerate}
\item The counit $\eps$ admits a natural section.
\item If moreover $G/H$ is finite, then the adjunction satisfies the projection formula (Def.\,\ref{Def:projection_formula}).
\end{enumerate}
\end{Lem}

\begin{proof}
Set $\xi:\Id_{\Shv(H;S)}\to \Res^G_H\Coind_H^G$ on an object $(N,\psi)\in\Shv(H;S)$ and on an open~$V$ to be the homomorphism $N(V) \to (\Coind_H^GN)(V)$, $n\mapsto (\xi(n)_t)_{t\in G}$ defined by the formula
\begin{equation}
\label{eq:xi-DG}%
\xi(n)_t := \left\{
\begin{array}{cl}
\psi_t(n) & \textrm{ if }t\in H
\\
0 & \textrm{ if }t\notin H.
\end{array}
\right.
\end{equation}
The $H$-equivariance of this morphism is easy to check and it is clearly natural in~$N$ and splits~$\eps$, since $\xi(n)_1=\psi_1(n)=n$. This proves~(a).

For the projection formula, let us assume that $G/H$ is finite. Recall that $\Res^G_1:\Shv(G;S)\to \cO_S\MMod$ detects isomorphisms. So it will suffice to prove that $\Res^G_1(\pi)$ is an isomorphism, where $\pi:(\Coind_H^GN)\otimes M \to \Coind_H^G(N\otimes \Res_H^G M)$ is the morphism of~\eqref{eq:right_projection_formula}, and this for every $N\in\Shv(H;S)$ and $M\in\Shv(G;S)$. One can verify that the morphism $\pi$ is given by (the sheafification of)
$$
\begin{array}{ccc}
(\Coind_H^GN)\otimes M & \too & \Coind_H^G(N\otimes \Res_H^G M)
\\
(s_t)_{t\in G}\otimes m & \longmapsto & \big(s_t\otimes \varphi_t(m)\big)_{t\in G}
\end{array}
$$
on simple tensors. We need to check that the above is an isomorphism of sheaves on~$S$, ignoring the $G$-equivariance. Let $R\subset G$ be a full set of representatives for~$G/H$. We claim that we have an isomorphism $\proj_R:\Res^G_1\Coind_H^G\isoto\prod_{r\in R}r^*$, which is given on every open $V\subseteq S$ by the composite of obvious inclusion and projection $(\Res^G_1\Coind_H^G N)(V)\hook \prod_{t\in G}t^*N(V) \onto \prod_{r\in R} r^*N(V)$, \ie~simply by $(s_t)_{t\in G}\mapsto (s_r)_{r\in R}$ on elements. Its inverse maps $(s_r)_{r\in R}$ to $(r^*\psi_h(s_r))_{t\in G}$ where each $t\in G$ is written in a unique way as $t=hr$ for some $h=h(t)\in H$ and some $r=r(t)\in R$. The defining property of the elements in~$(\Coind_H^GN)(V)$ shows that this is a bijection. Moreover, the following diagram of $\cO_S$-modules commutes:
$$
\xymatrix@R=2em@C=4em{\Res^G_1(\Coind_H^G N)\otimes \Res^G_1(M) \ar[d]^{\cong} \ar[r]_-{\proj\otimes 1}^-\simeq
&(\prod_{r\in R}r^* N)\otimes M \ar[d]^{\cong }_-{(\dagger)}
\\
\Res^G_1((\Coind_H^G N)\otimes M) \ar[d]^{\Res^G_1(\pi)}
& \prod_{r\in R}(r^* N\otimes M ) \ar[d]_\simeq^-{\prod_{r}\displaystyle \id\otimes\varphi_r}
\\
\Res^G_1\big(\Coind_H^G(N\otimes \Res^G_H M)\big) \ar[r]^-{\proj}_-{\simeq}
& \prod_{r\in R}r^* N\otimes r^* M.
}
$$
Note that the isomorphism~$(\dagger)$ uses finiteness of~$R\simeq G/H$ to commute product and tensor. We conclude that $\Res^G_1(\pi)$, hence $\pi$, is an isomorphism as wanted.
\end{proof}

\begin{proof}[Proof of Theorem~\ref{thm:intro_DG}]
	We can use Theorem~\ref{Thm:general} on the sheaf level (before deriving), since its hypotheses were checked in Lemma~\ref{lem:xi-proj-DG}. Let us identify the ring object $A=\Coind(\unit)$, where the $\otimes$-unit $\unit$ of $\Shv(H;S)$ is $\cO_S$ equipped with the trivial $H$-action: each $\psi_h$ is the isomorphism $h^\sharp:\cO_S\isoto h^*\cO_S$ given with the isomorphism of ringed spaces \mbox{$h:S\isoto S$}. Explicitly, $A(V)=\SET{(s_t)_t\in \prod_{t\in G}\cO_S(t V)}{s_{ht}=h^\sharp s_t\textrm{ for all }t\in G,\, h\in H}$, where $h^\sharp:\cO_S(t V)\isoto h^*\cO_S(t V)=\cO_S(ht V)$ is as above on the open~$t V$. Its multiplication~\eqref{eq:unit_monoid} is simply mapping $(s_t)_t\otimes (s'_t)_t$ to $(s_ts'_t)_t$. To show that this ring object~$A=\Coind(\unit)$ is isomorphic to the announced one $A^G_H=\oplus_{G/H}\cO_S$, we need to clarify the latter. On each open, if we name the basis~$\{e_\gamma\}_{\gamma\in G/H}$, it is given by $A^G_H(V)=\oplus_{\gamma\in G/H}\cO_S(V)\cdot e_\gamma$. The $G$-action $\varphi_g$ on $A^G_H$ simply maps $e_{\gamma}$ to $e_{g\gamma}$ and applies $g^\sharp:\cO_S(V)\isoto g^*\cO_S(V)$ on the coefficients. The isomorphism $A^G_H\isoto A$ can now be given on each open $V$ by mapping $e_\gamma$ to $(s_t)_t$ where $s_t=0$ if $t\notin \gamma$ and $s_t=1$ if $t\in \gamma$. An inverse is then given by mapping $(s_t)_t\in A(V)$ to $\sum_{\gamma\in G/H} s_\gamma e_\gamma$, where $s_\gamma:=(t^\sharp)\inv(s_t)\in \cO_S(V)$ for \emph{any} choice of $t\in\gamma$. The condition on $(s_t)_t\in A(V)$ shows that these coefficients $s_\gamma$ are well-defined.
Following the above isomorphism, one sees that the multiplication on~$A^G_H$ is given on the $\cO_S$-basis by the now familiar formula~\eqref{eq:mu}. Separability again follows from the $A^G_H,A^G_H$-linear morphism $\sigma:A^G_H\to A^G_H\otimes A^G_H$, defined on the $\cO_S$-basis by $e_\gamma\mapsto e_\gamma\otimes e_\gamma$.

	Finally, we need to derive the adjunction. This is easy since both functors $\Res^G_H$ and $\Coind_H^G$ are exact (the latter can be checked by postcomposing with $\Res^G_1$). Similarly, $A^G_H\otimes-$ is exact since $A^G_H$ is flat, the underlying $\cO_S$-module being a finite sum of copies of~$\cO_S$. It follows that all adjunctions, units, counits, sections thereof ($\xi$), etc, derive on the nose, without needing any resolution.
\end{proof}

\begin{Rem}
\label{rem:E-DG}%
By construction~\eqref{Eq:EM_monad}, the equivalence $E:\Der(H;S)\too A^G_H\MMod_{\Der(G;S)}$ is given explicitly by $E(N)=(\Coind_H^G(N), \Coind_H^G(\eps_N))$ for all $N\in \Shv(H;S)$, and similarly for complexes $N\in\Der(H;S)$, degreewise.
Under the isomorphism of monads $\pi: A^G_H\otimes(-)\stackrel{\sim}{\to} \bbA$, this $\cO_S(G/H)$-action becomes the morphism
\[
\xymatrix@R=1em{
 \cO_S(G/H)\otimes \Coind_H^G(N) \ar[r]^-{\pi} &
\Coind_H^G \Res^G_H \Coind_H^G ( N)
\ar[rr]^-{\Coind_H^G(\eps_N)} &&
\Coind_H^G(N)\,.
}
\]
It is given for every $\gamma\in G/H$ and every $(s_t)_{t}\in(\Coind_H^GN)(V)$ as in~\eqref{eq:Coind} by the formula $\gamma\otimes (s_t)_t\mapsto (s^\gamma_t)_t$ where $s^\gamma_t=s_t$ when $t\in\gamma$ and zero otherwise. The converse $E\inv$ can be described as in Lemma~\ref{Lem:sep_mon}.
\end{Rem}

%------------------------------------------------------------------------------%]]]
\bigbreak
\section{Counterexamples}
\label{se:counterexample}%
%------------------------------------------------------------------------------[[[
The goal of this section is to prove Theorem~\ref{thm:intro_counterex}. We consider a connected compact Lie group~$G$ and a non-trivial finite subgroup~$H\leq G$. We want to show that the restriction-coinduction adjunction $\SH(G) \adjto \SH(H)$ is not monadic. It suffices to show that the right adjoint $\Coind_H^G:\SH(H)\to \SH(G)$ is not faithful. This reduces to proving that the counit $\eps_D:\Res^G_H\Coind_H^G D\to D$ is not split for some object~$D\in \SH(H)$. This reduction is well-known: the morphism $\tau$ appearing in a distinguished triangle
$\csbull\otoo{\eps}\csbull\otoo{\tau}\csbull\too\csbull$
maps to zero under~$\Coind_H^G$ (since $\Coind_H^G(\eps)$ is split, by the unit-counit relation) but $\tau$ is not zero (since $\eps$ itself is not split).

First we reduce the problem to a statement about spaces.
\begin{Lem}\label{Lem:counterexample_derivedcounit}
	Let $\mathcal C$ and $\mathcal D$ be model categories
	in which every object is fibrant, let $d$ be a cofibrant object in $\mathcal D$, and let $f^* : \mathcal
	C \adjto \mathcal D :f_*$ be a Quillen adjunction.
	If the counit $\underline{\eps}_d : \LL f^* \RR f_* d \to d$ of the derived adjunction is split epi
	as a morphism in $\Ho \mathcal D$ then the original counit $\eps_d : f^*f_* d \to d$ is split epi up to homotopy:
	there is a map $\xi_d : d \to f^* f_* d$ in $\mathcal D$ such that $\eps_d \circ \xi_d$ is homotopic to $\id_d$.
\end{Lem}
\begin{proof}
	If $\gamma :\Gamma c \to c$ is the cofibrant replacement then $\LL f^* c = f^* \Gamma c$
	while $\RR f_*d = f_*d$ since every object is fibrant.
	Recall (\eg~from the proof of \cite[Thm.\,8.5.18]{Hirschhorn99}) that the
	derived adjunction $\LL f^* : \Ho \mathcal C \adjto \Ho \mathcal D : \RR f_*$ is given as follows:
    \begin{align*}
		\Ho \mathcal D (\LL f^*c,d) &= \Ho \mathcal D (f^* \Gamma c,d)
		= \pi(\Gamma f^* \Gamma c,\Gamma d)
		\cong \pi(\Gamma f^* \Gamma c,d)
		\cong \pi(f^* \Gamma c,d)\cong \\
		&\cong \pi(\Gamma c,f_* d)
		\cong \pi(\Gamma c,\Gamma f_* d)
		= \Ho \mathcal C (c,f_* d)
	    = \Ho \mathcal C (c,\RR f_* d).
	\end{align*}
	Taking $c = \RR f_*d$ and chasing the identity we find that the counit
	$\LL f^* \RR f_* d \to d$ is the homotopy class of a map
	$\underline{\eps}_d : \Gamma f^* \Gamma f_* d \to \Gamma d$ with the property that
	\[\xymatrix{
			\Gamma f^* \Gamma f_* d \ar[r]^{\gamma} \ar[d]_{\underline{\eps}_d} & f^* \Gamma f_* d \ar[r]^{f^* \gamma} & f^* f_* d \ar[d]^{\eps_d} \\
			\Gamma d \ar[rr]_{\gamma} && d
		}\]
	commutes up to homotopy.
	A section $d \to \LL f^* \RR f_* d$ of~$\underline{\eps}_d$ is the homotopy class of a map $\underline{\xi}_d :\Gamma d \to \Gamma f^* \Gamma f_* d$
	such that $\underline{\eps}_d \circ \underline{\xi}_d \sim \id_{\Gamma d}$.
	If $d$ is cofibrant then the cofibrant replacement $\gamma : \Gamma d \to d$ is a homotopy equivalence.
	For any quasi-inverse $\delta : d \to \Gamma d$ of~$\gamma$,
	the composite
	\[ \xymatrix@C=3em{ d \ar[r]^-{\delta}  &\Gamma d \ar[r]^-{\underline{\xi}_d}& \Gamma f^* \Gamma f_* d \ar[r]^{\gamma} & f^* \Gamma f_*d \ar[r]^{f^*\gamma} & f^*f_* d}\]
	defines a map $\xi_d : d \to f^* f_* d$
	such that $\eps_d \circ \xi_d \sim \id_d$.
\end{proof}
\begin{Lem}\label{Lem:counterexample_loops}
	Let $H$ be a closed subgroup of a compact Lie group $G$. If $D$ is an $H$-CW spectrum
	such that the counit $\underline{\eps}_D$ of the restriction-coinduction adjunction
	$\SH(G) \adjto \SH(H)$ is a split epi as a morphism in $\SH(H)$,
	then the space-level counit $\eps_{\Omega^\infty(D)}$ is a split epi up to homotopy in the
	category of based $H$-spaces.
\end{Lem}
\begin{proof}
	Let $G \mathcal S$ denote the category of $G$-spectra (in the sense of \cite{LMS86})
	so that $\SH(G) = \Ho G \mathcal S$
	and let $G\mathcal T$ denote the category of based $G$-spaces.
	For any closed subgroup $H \le G$, we have restriction-coinduction adjunctions at the level of spaces and at the level of spectra.
	In both cases, we denote coinduction by $F_H(G_+,-)$.
	One checks immediately from the definition  of $F_H(G_+,-):H\mathcal S \to G \mathcal S$
\cite[\S II.4, p.\,77]{LMS86}
	that
	$\Omega^\infty F_H(G_+,D) = F_H(G_+,\Omega^\infty D)$
	for any $H$-spectrum $D$ and that
	$\Omega^\infty (\eps_D) = \eps_{\Omega^\infty D}$ where
	$\eps$ denotes the counit of the restriction-coinduction
	adjunction.
	If $D$ is an $H$-CW spectrum such that the derived
	counit $\underline{\eps}_D$ in $\SH(H) = \Ho H \mathcal
	S$ is a split epi, then
	Lemma~\ref{Lem:counterexample_derivedcounit}
	implies that the counit $\eps_D$ in $H \mathcal S$ is a split epi up to homotopy.
	The functor $\Omega^\infty : H\mathcal S \to H \mathcal T$ preserves homotopy
	and so we conclude that
	$\Omega^\infty(\eps_D) = \eps_{\Omega^\infty D}$ is
	split epi up to homotopy in $H \mathcal T$.
\end{proof}
\begin{Lem}\label{Lem:counterexample_space}
	Let $H$ be a closed subgroup of a connected topological group $G$, and let~$X$ be a discrete based $H$-space.
	If there exists a continuous map $\xi : X \to F_H(G_+,X)$ such that the composite
	\[\xymatrix@1{ X \ar[r]^-{\xi} &F_H(G_+,X) \ar[r]^-{\eps_X} & X}\]
	is homotopic to $\id_X$ then $X$ is a trivial $H$-space.
\end{Lem}
\begin{proof}
	Let $x \in X$ and consider its image $\xi_x : G_+ \to X$.
	Homotopic maps to $X$ are equal since $X$ is discrete, so
	$\xi_x(e) = \eps_X(\xi_x) = x$.
	It follows that
	$\xi_x(g)=\xi_x(e)=x$ for all $g\in G$
	since $G$ is connected and $X$ is discrete.
	Hence, since $\xi_x$ is $H$-equivariant, $x=\xi_x(h) = h.\xi_x(e) = h.x$ for all $h\in H$.
\end{proof}
\begin{proof}[Proof of Theorem~\ref{thm:intro_counterex}]
	Armed with Lemma~\ref{Lem:counterexample_loops} and Lemma~\ref{Lem:counterexample_space}
we need only
show that there exists an $H$-spectrum~$D$ such that
the based $H$-space $\Omega^\infty D$ is \mbox{$H$-homotopy} equivalent to
a \emph{discrete} based $H$-space with \emph{non-trivial} $H$-action.
Indeed, if $\gamma : \Gamma D \to D$ denotes an $H$-CW-approximation then Lemma~\ref{Lem:counterexample_loops} implies
that the counit of $\Omega^\infty \Gamma D$ splits up to homotopy in the category of based $H$-spaces.
The map $\Omega^\infty \gamma : \Omega^\infty \Gamma D \to \Omega^\infty D$ is a weak $H$-equivalence between spaces
having the \mbox{$H$-homotopy} type of an $H$-CW-complex
and hence is an $H$-homotopy equivalence by Whitehead's theorem.
Moreover, it is clear that the splitting up to homotopy of the counit for a based $H$-space $X$ implies
the splitting of the counit up to homotopy for any based $H$-space $H$-homotopy equivalent to $X$, so we can
apply Lemma~\ref{Lem:counterexample_space} to obtain our contradiction.

Note that we can't just take $D=\Sigma^\infty H_+$ since $\Omega^\infty \Sigma^\infty H_+$ is $Q(H_+) = \prod_H Q(S^0)$
rather than $H_+$.
Nevertheless, recall that for any $H$-Mackey functor $\mathcal M$, there exists an equivariant Eilenberg-MacLane spectrum $\mathcal H \mathcal M$
having the property that
\[\underline{\pi}_q(\mathcal H \mathcal M) = \begin{cases} \mathcal M & \text{if } q=0 \\ 0 & \text{if } q \neq 0 \end{cases}\]
so that $\Omega^\infty \mathcal H\mathcal M$ is an equivariant Eilenberg-MacLane space of type $K(\mathcal M,0)$.
The original construction \cite{LewisMayMcClure81} is elegant but indirect (involving a Brown representability argument).
A very concrete construction has been provided by
\cite{dosSantosNie08pp,dosSantos03}.
Applied to the ``fixed point'' Mackey functor $\underline{M}(H/K) := M^K$ associated to a $\mathbb Z H$-module $M$
one
obtains an $H$-spectrum $\mathcal H\underline{M}$
whose zeroth space
$\Omega^\infty \mathcal H\underline{M}$ is $H$-homotopy equivalent to $M$ regarded as a discrete based $H$-space (having $0$ as the base point).
Taking $M = \mathbb Z H$ to be the regular representation, we thus obtain a discrete non-trivial based $H$-space---provided $H$ itself is non-trivial---and this completes the proof.
\end{proof}

%==========================================================================================
\bibliographystyle{alpha}\bibliography{TG-articles}\end{document}